\documentclass[a4paper,12pt]{article}
\usepackage{cmap}
\usepackage[T1]{fontenc}
\usepackage[utf8]{inputenc}
\usepackage{amsmath,amsfonts,amssymb,amsthm}
\usepackage{tikz,graphicx,pgfplots}
\usepackage{bm}
\usepackage{float}
\allowdisplaybreaks
\usepackage{mathtools}
\mathtoolsset{showonlyrefs}

\usepackage{secdot}

\usepackage[labelsep=period]{caption}

\usepackage{enumitem}
\setlist{nosep}

\numberwithin{equation}{section}
\newtheorem{lemma}{Lemma}[section]
\newtheorem{remark}{Remark}[section]
\newtheorem{theorem}{Theorem}[section]
\newtheorem{corollary}{Corollary}[section]
\let\OLDthebibliography\thebibliography
\renewcommand\thebibliography[1]{
	\OLDthebibliography{#1}
	\setlength{\parskip}{1pt}
	\setlength{\itemsep}{1pt plus 0.3ex}
}

\usepackage[colorlinks=true,linktocpage,pdfpagelabels,
bookmarksnumbered,bookmarksopen]{hyperref}
\definecolor{ForestGreen}{rgb}{0.1,0.6,0.05}
\definecolor{EgyptBlue}{rgb}{0.063,0.1,0.6}
\hypersetup{
	colorlinks=true,
	linkcolor=EgyptBlue,         
	citecolor=ForestGreen,
	urlcolor=olive
}

\usepackage{accents}

\def\R{\mathbb R}

 \title{Entire solutions to the Swift--Hohenberg equation\\ via variational 
approach}
 \author{S.B.~Kolonitskii\footnote{Saint-Petersburg State Electrotechnical 
University ``LETI'', Russia; email: {sbkolonitskii@etu.ru}}, 
\setcounter{footnote}{6}
 L.M.~Lerman\footnote{HSE University and N.I. Lobachevsky State University of 
Nizhny Novgorod, Russia; email: {lermanl@mm.unn.ru}. The work of this author 
is supported by the Ministry of Science and Higher Education of the Russian 
Federation, agreement 0729-2020-0036.}, 
\setcounter{footnote}{3}
 A.I.~Nazarov\footnote{PDMI RAS and St. Petersburg State University, Russia; 
email: al.il.nazarov@gmail.com}}
\date{}

\begin{document}
 \maketitle

\section{Introduction and set up}

The Swift--Hohenberg (SH) equation
\begin{equation}\label{MainEq}
u_t = \alpha u + \beta u^2 - u^3 - (\Delta+1)^2u
\end{equation}
($\Delta$ is the conventional spatial Laplacian)
models many phenomena in various branches of physics, chemistry and ecology
\cite{SH,TGM,VM,Knobloch,Meron}. Therefore its study is of
interest both for the applications and from the viewpoint of its mathematical interest as a semi-linear PDE having a rich set of solutions with different structures \cite{TGM,Doel,IR,spots,KLSh}. This equation is of the gradient type \cite{Fife}, therefore  solutions of the corresponding stationary equation
\begin{equation}
  \label{Swift--Hohenberg-equation}
  (\Delta + 1)^2 u - \alpha u - \beta u^2 + u^3=0
\end{equation}
are of the primary interest, they can be found by variational methods.

When we deal with some model PDE of the evolutionary type arising in physics, two main goals can be of interest. The first is a physical point of view: which temporary stable
solutions can this equation have, what is their spatial structure, different problems concerning asymptotic behavior of solutions, the bifurcation problems for solutions found. At this approach temporary unstable solutions are out of interest. The situation and
interests have changed with the development of the theory of differential equations and arising the theory of dynamical systems, when the main interest has shifted from the study of individual solutions to the structure of a system in the whole or within large domains. It appeared that unstable solutions also matter. 

The today's situation in PDEs studies on the physical level remind, in a sense, the old history in mechanics, when main interest was the study of stable solutions. Nevertheless, we understand, that from the future vision of the theory, it is very important to understand the behavior and structure for all solutions of the PDE given, how complicated can they be, what are the mathematical mechanisms responsible for their existence and changes of their structure, as parameters vary. This is the direction of research which we follow in this paper. We would like to understand what are the solutions of the stationary SH equation when it is considered on the whole $\R^n$. It is necessary to emphasize that a great deal of work was done to understand this problem at the various level of rigor. Since it is impossible to describe more or less completely all the work done here, we mention only surveys \cite{Knobloch,Knob,Hoyle,CG,Nep}, in which a rich lists of references can be found.

The equation \eqref{Swift--Hohenberg-equation} is considered as a representative model equation, being not very complicated as it possesses a variational structure, but sufficiently complicated to have many solutions with different patterns. This is why we try to study its solutions as complete as possible using the tools we possess. We emphasize from the very beginning that we approach to the problem purely mathematically without paying so far the attention to whether the solutions found are relevant for the physical problems where the equation was derived. We hope to return to this problem at the next iteration of the study. One needs to emphasize that the study of solutions in the whole space is a rather hard problem and not many methods were developed here, especially if one deals with equations of higher order (greater than two).

The spatially one-dimensional (1D) case of SH equation is relatively well-studied. The existence and a genesis of some stationary solutions (localized, periodic, almost periodic, others) are rigorously substantiated \cite{GL,BK,KLSh}. Namely, the homogeneous stationary solution of SH equation corresponds to the equilibrium at the origin for the related reversible Hamiltonian system in $\R^4$ obtained after introducing new variables, and the main bifurcation of this equilibrium is the Hamiltonian Hopf bifurcation \cite{Meer,Schmidt}. For instance, spatially localized solutions correspond to homoclinic orbits arising from the equilibrium at the origin for this Hamiltonian system. It is worth noting that in the 1D case there is a region in the parameter plane $(\alpha,\beta)$, for points from there no localized solutions exist, see \cite{BGL}.

The Hamiltonian Hopf bifurcation provides a local mechanism of the creation homoclinic orbits and therefore, for solutions to SH equation from the homogeneous state.
There are also nonlocal mechanisms of their creation, when variation of parameters leads to the appearance of tangent homoclinic orbits which then are destroyed and form pairs of transverse homoclinic orbits. This process is similar to an appearance of a periodic orbit from ``nothing'' through a tangent bifurcation. One more mechanism of creating homoclinic orbits is through a formation of symmetric heteroclinic connections with some saddle periodic orbit lying in the same level of the Hamiltonian, or with another saddle-focus, arising for parameters along some curve in the parameter plane corresponding to the appearance of the second saddle-focus in the same level \cite{heter,KLSh}.

The proof of the local creation of localized solutions uses the reversibility of the Hamiltonian system and results of \cite{IP}. For this two-parametric Hamiltonian system double imaginary eigenvalues of the equilibrium at the origin correspond to the axis $\alpha=0$ in the parameter plane $(\alpha,\beta)$. The type of bifurcation is determined by the sign of a certain coefficient $A$, which alters at the point $\beta=\beta^\star\equiv\sqrt{27/38}$. The structure of the system in a neighborhood of this point was investigated in \cite{GL} for the truncated (integrable) normal form of the sixth order. It was shown there that the passage from $A<0$ ($\beta<\beta^\star$), when no local homoclinic orbits arise, to the case $A>0$ ($\beta>\beta^\star$), when they do arise, is accompanied with the formation of two one-parameter families of heteroclinic orbits, going from the saddle-focus to a saddle periodic orbit lying both in the same level of the Hamiltonian.


In the full (non-integrable) system the formation of local homoclinic orbits of the saddle-focus is blurred and is accompanied by the formation first of four tangent heteroclinic orbits, and there is a layer in the parameter plane, where countably many bifurcations occur, and only after crossing this layer the homoclinic orbits appear. The numerical justification of this picture was confirmed in \cite{Champ} for a similar model.

The spatially multidimensonal case is less investigated. The equation \eqref{Swift--Hohenberg-equation} in $\R^n$ is invariant w.r.t. the action of the group $O(n)$ of rotations of the Euclidean space $\R^n$. For $n=2$, the existence of its radial (rotationally invariant) localized solutions is rather well studied \cite{KLSh,spots,Sand}. The numerical experiments also demonstrate the existence of non-radial localized stationary patterns to this equation \cite{spots,Sand} which are invariant w.r.t. some finite subgroup of $O(2)$. Till recently most investigations in the two-dimensional geometry rely either on finite-dimensional approximations or some plausible hypotheses whose validity is not yet proven with mathematical level of rigor \cite{ALBKS,KL}. In the case $\alpha > 0$, some rigorous results were obtained in \cite{BIS,IR,Iooss}, where the existence of quasipatterns was proven using methods close to those in KAM theory. All this makes other approaches necessary and desirable, since they may lead to other possible mechanisms of finding and constructing solutions.

SH equation with $n\ge 3$ is almost not explored. For other types of equations (systems) met in chemistry (reaction-diffusion systems), laser study etc{.} there are some numerical results \cite{deWit,Stalun,Tlidi} but as we aware of, no rigorous mathematical results are known.


In this paper we develop a variational approach which allows us to construct new classes of solutions to SH equation with $\alpha < 0$ in $\R^n$ for any $n \leq 7$ in a unified way.  

Variational methods for differential equations are well known. For instance, homoclinic and periodic solutions for the systems of second order ordinary differential equations periodic in $x$ were obtained in papers starting since the work by Rabinowitz \cite{Rabinowitz} and many others after that \cite{Rab1,Rab2,Rab3}. Important novel methods and results were obtained by Ser\'e \cite{Sere1,Sere2}. This approach was extended to systems of elliptic equations periodic w.r.t. independent variables \cite{Rab1,Rab3}.

Our version of variational approach based on symmetry considerations was developed in \cite{LNN}, where entire bounded solutions with various types of symmetries to the simplest semilinear elliptic equation 
\begin{equation}
 \label{simplest_equation}
 \Delta u - u + u^3=0  \quad \mbox{in} \quad \mathbb R^n, \quad n = 2, 3,
\end{equation}
were constructed using the concentration-compactness principle by P.-L. Lions. Moreover, similar results were obtained in \cite{LNN} for the quasilinear equation
\begin{equation}
\Delta_pu - |u|^{p-2}u + |u|^{q-2}u = 0 \quad \mbox{in} \quad \mathbb R^n, \quad 1 < p < \infty,
\label{GenerEq}
\end{equation}
driven by $p$-Laplacian $\Delta_p u \equiv \mathrm{div}(|\nabla u|^{p-2}u)$, with a superlinear and subcritical exponent $q$. 

In a recent paper \cite{NSch}, the methods of \cite{LNN}  were modified and extended  to the non-local elliptic equations driven by fractional Laplacian $(-\Delta)^s$, $0<s<1$.

 In this paper we obtain a series of periodic solutions to the equation \eqref{Swift--Hohenberg-equation} with certain additional symmetries. First, we consider a boundary value problem for this equation in a bounded convex polyhedron $\Omega \subset \R^n$, boundary conditions being the \textbf{half-Neumann boundary conditions}
 \begin{equation}
  \label{Navier-2}
  \frac {\partial u} {\partial \bf{n}} = \frac {\partial} {\partial \bf{n}} (\Delta u) = 0 \quad\mbox{on}\quad \partial\Omega,
 \end{equation}
 also called the \textbf{``sliding wall'' boundary conditions}. 

In Section \ref{S:aux} we introduce the appropriate energy space and the energy functional such that its critical points are weak solutions of the problem \eqref{Swift--Hohenberg-equation}--\eqref{Navier-2}. This Section also contains some auxiliary lemmata. 

Section \ref{S:variational_solutions} is the core of our paper. We thoroughly analyse the Nehari manifold\footnote{In this problem the Nehari manifold has a more complicated structure compared to Nehari manifold for the equation \eqref{simplest_equation}, see Remark \ref{bubbles} below.} for the problem \eqref{Swift--Hohenberg-equation}--\eqref{Navier-2} to establish the sufficient conditions of existence of a so-called ``ridge-Nehari solution'' in terms of parameters $(\alpha,\beta)$. We stress that for domains
\begin{equation}
\label{Omega_R}
     \Omega_R=\{Rx\,\big| \, x\in\Omega\}, \qquad R\ge1,
 \end{equation}
these sufficient conditions do not depend on the stretching factor $R$. 

Generally, we cannot rule out the possibility that a solution we obtained is in fact a constant or is generated by a solution in lower dimension (i.e. is independent of some variables). However, in the last subsection of Section \ref{S:variational_solutions} we prove that ridge-Nehari solutions in $\Omega_R$ have a nontrivial dependence on all variables if $R$ is large enough.

In Section \ref{S:periodic_solutions}, similar to \cite{LNN, NSch}, we introduce the concept of fundamental domain and demonstrate that even reflections of a ridge-Nehari solution\footnote{ Extension of a solution by even reflections is possible just due to  boundary conditions \eqref{Navier-2}.} in a fundamental domain generate a solution in the whole space with corresponding symmetries. 

In Section \ref{S:skew_periodic_solutions} we construct a family of solutions with a different structure.

In Appendix, proofs of auxiliary statements are provided. In particular, we prove that the eigenfunctions of the Neumann Laplacian in a convex polyhedron $\Omega$ are orthogonal not only in $L_2(\Omega)$ and in $W^1_2(\Omega)$ (that is well known for arbitrary Lipschitz domain) but also in $W^2_2(\Omega)$.\footnote{ The same is true for the eigenfunctions of the Dirichlet Laplacian. }

\paragraph{Acknowledgements.} We are grateful to A.P. Shcheglova who has pictured Figure~\ref{fig:1}.

 \section{Auxiliary statements}
 \label{S:aux}

We introduce the {\bf energy space}
 \begin{equation}
  {\cal W}(\Omega) = \left \lbrace v \in W^2_2(\Omega) : \frac {\partial u}{\partial \bf{n}} = 0 \mbox{ on } \partial\Omega \right \rbrace\tag{WN2}\label{WNav_2}
 \end{equation}
 and define the {\bf energy functional} 
  \begin{equation}
   \label{energy-functional}
   E[u] = \int\limits_\Omega \left [ \vphantom{\int\limits} \frac 1 2(\Delta u + u)^2 - \frac 1 2 \alpha u^2 - \frac 1 3 \beta u^3 + \frac 1 4 u^4  \right ] \,\,dx.
  \end{equation}
 \begin{remark} 
  \label{compact_embedding_remark}
  By the Sobolev embedding theorem for $n \leq 8$ the space $W^2_2(\Omega)$ is continuously embedded into $L_3(\Omega)$ and $L_4(\Omega)$, and thus $E[u]$ is well-defined on ${\cal W}(\Omega)$. We consider dimensions $n \leq 7$, as in this case these embeddings are compact, and  this is crucial for the proof. However, some statements hold for $n=8$.  
 \end{remark}
 
 By standard variational argument, any critical point of $E[u]$ on ${\cal W}(\Omega)$ satisfies the integral identity
 \begin{equation}
  \label{weak_SH_solution}
  dE[u]h\equiv\int\limits_\Omega \left [ \vphantom{\int\limits} \Delta u \Delta h - 2 (\nabla u, \nabla h) + (1-\alpha) uh - \beta u^2 h + u^3 h \right ] \,\,dx = 0.
 \end{equation}
 for all $h \in {\cal W}(\Omega)$. Therefore, $u$ is a weak (Sobolev) solution of \eqref{Swift--Hohenberg-equation} in $\Omega$, and the second condition in \eqref{Navier-2} is a natural boundary condition. 
 
 From now on, we suppose that $\alpha < 0$. Also, since the replacement $u\mapsto -u$, $\beta\mapsto -\beta$ preserves the equation \eqref{Swift--Hohenberg-equation} and the energy functional, we can assume without loss of generality that $\beta>0$.

  \begin{lemma}
   \label{equivalent_norm_lemma}
   Let $\Omega$ be a convex polyhedron. The quadratic part of \eqref{energy-functional}
   \begin{equation}
    Q[u] = \int\limits_\Omega \Big( (\Delta u + u)^2 - \alpha u^2 \Big) \,\,dx
   \end{equation}
   defines a norm on ${\cal W}(\Omega)$ that is equivalent to the standard norm in $W_2^2(\Omega)$. Namely,
   \begin{equation}
    c \|u\|_{W_2^2}^2 \le Q[u] \le C \|u\|_{W_2^2}^2
   \end{equation}
   where constants $c$ and $C$ depend only on $\alpha$ (in particular, they do not depend on $\Omega$).
  \end{lemma}

  Proof is given in Appendix 1.
  \medskip

It is useful to define 
  \begin{equation}
   \label{Sobolev_constant_Neumann}
   S_m = S_m(\Omega) = \inf\limits_{u \in {\cal W}(\Omega)} \frac {(Q[u])^{\frac 1 2}}{\Bigr |\int\limits_\Omega u^m \,\,dx\Bigr |^{\frac 1 m}}, \qquad m =2, 3, 4. 
  \end{equation}
  These infima are achieved, due to Remark \ref{compact_embedding_remark}. 

  \begin{lemma}
   \label{Sobolev_constants_bounded_lemma}
   For domains $\Omega_R$ defined in \eqref{Omega_R}, the quantities $S_m(\Omega_R)$ are bounded from above and from below as $R \to \infty$.
  \end{lemma}
  \begin{proof}
   Upper bounds are obtained by taking a sample function with compact support. 
   
   Let us now prove the lower bounds. Let $\Pi_R$ be the extension operator from $W_2^2(\Omega_R)$ to $W_2^2(\R^n)$. Then we have ($c$ is the constant from Lemma \ref{equivalent_norm_lemma})
   \begin{equation}
   \aligned
    Q[u]\ge c \|u\|_{W_2^2(\Omega_R)}^2 &\ge c \|\Pi_R\|^{-2} \|\Pi_R u\|_{W_2^2(\R^n)}^2 \\
    &\ge cS_m^2(\mathbb R^n) \|\Pi_R\|^{-2} \|\Pi_R u\|_{L_m(\R^n)}^2 \ge cS_m^2(\mathbb R^n) \|\Pi_R\|^{-2} \|u\|_{L_m(\Omega_R)}^2.
   \endaligned
   \end{equation}
   It follows from the proof of \cite[Theorem 6.5]{stein} that $\|\Pi_R\|$ is bounded by the Lipschitz constant of the boundary, which is invariant under scaling, and geometry of $\Omega$. The latter is comprised of the (maximal) number of overlapping Lipschitz maps, which is also invariant under scaling, and the (minimal) diameter of Lipschitz maps, which increases under dilation. Thus $\|\Pi_R\| \le C$ for all $R \ge 1$. This completes the proof.
  \end{proof}


  \begin{remark}
   \label{S_two_bounds}
   Estimating $\big ( \Delta u + u \big )^2 \ge 0$, we obtain that $S_2 \ge \sqrt{-\alpha}$ for any $\Omega$. Substituting a constant into \eqref{Sobolev_constant_Neumann}, we obtain that $S_2 \le \sqrt{1 - \alpha}$ for any $\Omega$.
  \end{remark}

 \section{Variational solutions of Swift--Hohenberg equation with half-Neumann boundary condition}
  \label{S:variational_solutions}
   Together with energy functional $E[u]$, consider the following functionals:
   \begin{align}
    L[u]&=dE[u]u = Q[u] - \beta \int\limits_\Omega u^3 \,\,dx + \int\limits_\Omega u^4 \,\,dx;\\
    H[u]&=dL[u]u = 2 Q[u] - 3 \beta \int\limits_\Omega u^3 \,\,dx + 4 \int\limits_\Omega u^4 \,\,dx;\\
    K_0[u]&=E[u] - \frac 1 2 L[u] = \frac 1 6\, \beta \int\limits_\Omega u^3 \,dx - \frac 1 4 \int\limits_\Omega u^4 \,dx;\\
    K_1[u] &= L[u] - \frac 1 2 H[u] = \frac 1 2 \beta \int\limits_\Omega u^3 \,dx - \int\limits_\Omega u^4 \,dx.
   \end{align}

   \begin{lemma}
    \label{functionals_continuity_lemma}
    Let $n \leq 7$. Then:
    \begin{enumerate}
     \item Functionals $E[u]$, $L[u]$ and $H[u]$ are well-defined, continuous and weakly lower semicontinuous on ${\cal W}(\Omega)$.
     \item Functionals $K_0[u]$ and $K_1[u]$ are weakly continuous on ${\cal W}(\Omega)$.
     \item Let $v \in {\cal W}(\Omega) \setminus \lbrace 0 \rbrace$ and define a fibration $\varphi_v(t) = E[t v]$. Then
     \begin{equation}\label{fibration_derivatives}
      \aligned   
      \varphi'_v(t)&= t^{-1} L[tv];\\
      \varphi''_v(t)&={t^{-2}} H[tv] - t^{-2} L[tv].
     \endaligned
     \end{equation}
    \end{enumerate}
   \end{lemma}
   \begin{proof}
    First, observe that quadratic terms of functionals $E[u]$, $L[u]$ and $H[u]$ are well-defined and continuous in ${\cal W}(\Omega)$. They are also convex, so they are automatically weakly lower-semicontinuous.
    
    The cubic and quartic terms are obviously continuous in $L_3$ norm (respectively, $L_4$ norm). Per Remark \ref{compact_embedding_remark} ${\cal W}(\Omega)$ is compactly embedded into $L_3(\Omega)$ and $L_4(\Omega)$, and thus these terms are well-defined and weakly continuous on ${\cal W}(\Omega)$. This proves claims 1 and 2. The third claim follows from direct calculations.   
   \end{proof}

   \begin{remark}
    For $n=8$ the embedding of ${\cal W}(\Omega)$ into $L_4(\Omega)$ is continuous, but not compact. In this case the quartic term is not weakly continuous, and so are functionals $K_0[u]$ and $K_1[u]$.
   \end{remark}

   \begin{lemma}
    \label{energy_coercive_lemma}
    Functional $E[u]$ is coercive on ${\cal W}(\Omega)$.
   \end{lemma}
   \begin{proof}
    By Lemma \ref{equivalent_norm_lemma} the quadratic term is equivalent to a square of the norm in  ${\cal W}(\Omega)$. 
    The cubic term can be estimated by the Young inequality as
    \begin{equation}
     \frac \beta 3 \int\limits_\Omega u^3 \,dx \le \frac 1 4 \int\limits_\Omega u^4 \,dx + \frac 1{12} |\Omega| \beta^4,
    \end{equation}
    Then the following estimate holds:
    \begin{equation}
     E[u] \ge c \|u\|^2_{{\cal W}(\Omega)} - \frac \beta 3 \int\limits_\Omega u^3 \,dx + \frac 14 \int\limits_\Omega u^4 \,dx \ge c \|u\|^2_{{\cal W}(\Omega)} - \frac 1{12} |\Omega| \beta^4, 
    \end{equation}
    which tends to infinity as $\|u\|_{{\cal W}(\Omega)} \to \infty$. 
   \end{proof}

  \subsection{Constant solutions}
  \label{ss:constant}

   Lemmata \ref{functionals_continuity_lemma} and \ref{energy_coercive_lemma} imply that the functional $E[u]$ has a global minimizer in ${\cal W}(\Omega)$, see, e.g., \cite[Theorem 26.8]{FK}, which is a (weak) solution to the boundary value problem \eqref{Swift--Hohenberg-equation}--\eqref{Navier-2}. However, we cannot exclude that this solution is trivial, say, constant. Since for $\beta=0$ the energy functional is strictly convex, the same holds for small $\beta>0$,\footnote{The quantitative version of this statement is given in Corollary \ref{convex}.} and therefore, a unique critical point of $E[u]$ is the global minimizer $u \equiv 0$. Moreover, we conjecture that for any $\beta>0$ the global minimizer of $E[u]$ in ${\cal W}(\Omega_R)$ in fact is a constant provided $R$ is large enough. So, we are interested in other critical points of $E[u]$.

First of all, we consider the constant solutions. If $u\equiv c\ne0$ is a solution to the problem \eqref{Swift--Hohenberg-equation}--\eqref{Navier-2} then
   \begin{equation}
    \label{c_pm_equation}
    (1-\alpha) - \beta c + c^2 = 0.
   \end{equation}
  From now on, we assume that $\beta > 2\sqrt{1-\alpha}$. In this case \eqref{c_pm_equation} has two solutions $c_+>c_->0$. 
   
   The energy second variations at constant solutions equal
   \begin{equation}
   \label{second-var-at-const}
    \aligned
    d^2E[0] (h,h) = &\int\limits_\Omega (\Delta h + h)^2 \,dx - \alpha \int\limits_\Omega h^2 \,dx;\\ 
    d^2E[c_\pm] (h,h) = &\int\limits_\Omega (\Delta h + h)^2 \,dx +  \left [ -\alpha -2 \beta c_\pm + 3 c_\pm^2  \right ] \int\limits_\Omega h^2 \,dx.
   \endaligned
   \end{equation}
Thus, zero is always a (strict) local minimum of the energy whereas $c_-$ is a saddle-like critical point, as it is a local maximum in the direction of $h=const$.

Further, since $c_\pm$ satisfy \eqref{c_pm_equation}, we have
   \begin{equation}
    m_\pm := -\alpha - 2 \beta c_\pm + 3 c_\pm^2 = -3 + 2 \alpha + \beta c_\pm. 
   \end{equation}
   
   If $\beta^2 > \frac 9 2(1-\alpha)$, we have
   \begin{equation}
   \label{c+}
    c_+ = \frac {\beta + \sqrt{\beta^2 - 4(1-\alpha)}} 2 > \frac {\sqrt{\frac 9 2 (1-\alpha)}+\sqrt{\frac 1 2 (1-\alpha)}}{2} = 2 \sqrt{\frac {1-\alpha} 2};
   \end{equation}
   \begin{equation}
    m_+ = -3 + 2\alpha + \beta c_+ > -3 + 2\alpha + 3 \sqrt{\frac {1-\alpha} 2} \cdot 2 \sqrt{\frac {1-\alpha} 2} = -\alpha.
   \end{equation}
   So, \eqref{second-var-at-const} implies $d^2E[c_+](h,h)>Q[h]$, and $c_+$ is also a strict local minimum. Moreover, by \eqref{c_pm_equation} its energy is
   \begin{equation}
    E[c_+] = \frac 1 2(1-\alpha) c_+^2 - \frac 1 3 \beta c_+^3 + \frac 1 4 c_+^4 = \frac 1{12}\, c_+^2 \left ( 2(1-\alpha) - c_+^2 \right ).
   \end{equation}By \eqref{c+} we have $E[c_+]< 0$ and therefore zero is not a global minimizer.

  \subsection{Nehari manifold and fibrations}

   Consider the Nehari manifold
   \begin{equation}
    \mathcal N(\Omega) = \left \lbrace v \in {\cal W}(\Omega) \setminus \lbrace 0 \rbrace : L[v] = 0 \right \rbrace.
   \end{equation}
   
   \begin{lemma}
    \label{fibration_classification_lemma}
    Let $0 \not= v \in {\cal W}(\Omega)$. Assume that $\int\limits_\Omega v^3 \,dx \ge 0$. Then there are no negative scaling factors $t$ such that $t v \in \mathcal N(\Omega)$, and exactly one of the following alternatives holds:
    \begin{enumerate}
     \item there are no positive scaling factors $t$ such that $t v \in \mathcal N(\Omega)$. We call such fibrations \textbf{monotonous}.
     \item there exists a unique scaling factor $t>0$ such that $t v \in \mathcal N(\Omega)$. In this case $H[tv]=0$. We call such fibrations \textbf{degenerate monotonous} or \textbf{degenerate}.
     \item there exist exactly two scaling factors $t_2 > t_1 > 0$ such that $t_1 v, t_2 v \in \mathcal N(\Omega)$. In this case we have $H[t_1 v]<0$ and $H[t_2 v]>0$. We call such fibrations \textbf{non-monotonous}.
    \end{enumerate}
   \end{lemma}
   \begin{proof}
    Consider the fibration $\varphi_v(t) = E[tv]$. It is a polynomial of degree 4, and thus must have at least one and no more than three critical points. As $Q[v] > 0$, the quadratic coefficient is positive, thus $t=0$ is always a nondegenerate minimum of fibration. Under assumptions of lemma $\varphi_v'(t) < 0$ for all $t < 0$. Thus if any nonzero critical point exists, it must be positive.

    At critical points of fibration $\varphi_v$ we have $H[tv] = t^{-2} \varphi_v''(t)$ by \eqref{fibration_derivatives}.  
   \end{proof}
   
   \begin{remark}
   \label{bubbles}
For the simplest equation \eqref{simplest_equation} with Neumann boundary conditions, corresponding energy functional is
\begin{equation}
\widetilde E[v] = \int\limits_\Omega \left [ \vphantom{\int\limits} \frac 1 2\,|\nabla v|^2 + \frac 1 2\,v^2 - \frac 1 4\, v^4  \right ] \,dx, \qquad v \in W^1_2(\Omega).
\end{equation}
It is well known that the Nehari manifold in this case is a ``bubble around the origin''. This means that the interior of the Nehari manifold is star-shaped: for any $v \in W^1_2(\Omega) \setminus \lbrace 0 \rbrace$ there exists a unique positive scaling factor $t$ such that $t v \in \mathcal N(\Omega)$.
   
In contrast, Lemma \ref{fibration_classification_lemma} shows that
the Nehari manifold for the problem \eqref{Swift--Hohenberg-equation}--\eqref{Navier-2} (if it is not empty) consists of one or several ``bubbles'' neither of which encompasses the origin.

From now on, for any fibration $\varphi_v$ we assume (changing the sign of $v$ if necessary) that $\int\limits_\Omega v^3 \,dx \ge 0$ and consider only nonnegative arguments of fibration $t \in [0,+\infty)$.
   \end{remark}

By Lemma \ref{fibration_classification_lemma}, any bubble of Nehari manifold contains two distinct parts:
   \begin{itemize}
    \item ``valley'' corresponding to the greater scaling factor $t_2$ and characterized by inequality $H[v] > 0$, and
    \item ``ridge'' corresponding to the smaller scaling factor $t_1$ and characterized by inequality $H[v] < 0$.
   \end{itemize}
\medskip
   
Now we can introduce the sets: 
   \begin{itemize}
    \item the ridge part of Nehari manifold
     \begin{equation}
      \mathcal N_-(\Omega) = \left \lbrace v \in {\cal W}(\Omega) \setminus \lbrace 0 \rbrace : L[v] = 0; H[v] < 0 \right \rbrace;
     \end{equation}
    \item the valley part of Nehari manifold
     \begin{equation}
      \mathcal N_+(\Omega) = \left \lbrace v \in {\cal W}(\Omega) \setminus \lbrace 0 \rbrace : L[v] = 0; H[v] > 0 \right \rbrace;
     \end{equation} 
    \item the ``end of ridge'' set
     \begin{equation}
      \mathcal N_0(\Omega) = \left \lbrace v \in {\cal W}(\Omega) \setminus \lbrace 0 \rbrace : L[v] = 0; H[v] = 0 \right \rbrace;
     \end{equation} 
    \item the ``rear slope'' set
     \begin{equation}
      \mathcal M_-(\Omega) = \left \lbrace v \in {\cal W}(\Omega) \setminus \lbrace 0 \rbrace : L[v] \le 0; H[v] \le 0; K_1[v] \ge 0 \right \rbrace.
     \end{equation}
   \end{itemize}

   
   \begin{remark}
    \begin{enumerate}
     \item If $v \in \mathcal N_-(\Omega)$ then the fibration $\varphi_v(t) = E[t v]$ is non-monotonous and $t=1$ is its (local) nondegenerate maximum.
     \item If $v \in \mathcal N_0(\Omega)$ then the fibration $\varphi_v(t) = E[t v]$ is degenerate. 
    \end{enumerate}   
   \end{remark}
    
   \begin{remark}
    Recall that we assume $\beta>2 \sqrt{1-\alpha}$, and constant solutions $c_\pm$ exist. Obviously, $c_-\in\mathcal N_-(\Omega)$ and $c_+\in\mathcal N_+(\Omega)$.
   \end{remark}
   
  \begin{lemma}
  \label{coercitivity_at_zero_lemma}
    \begin{enumerate}
     \item Let $v$ be such that its fibration is monotonous. Then $E[v] \ge \frac 1 {18} Q[v]$.
     \item Let $v \in \mathcal N_-(\Omega)\cup \mathcal N_0(\Omega)$. Then $E[t v] \ge \frac 1 {12} Q[t v]$ for any $t \in [0,1]$.\smallskip
    \end{enumerate}
   \end{lemma}
   \begin{proof}

Denote for brevity
 \begin{equation}
 \label{QBD}
   Q := Q[v]; \qquad B: = \beta \int\limits_\Omega v^3 \,dx; \qquad D: = \int\limits_\Omega v^4 \,dx.
 \end{equation}

    1. It follows from \eqref{fibration_derivatives} that
    the monotonicity of fibration of $v$ is equivalent to 
    \begin{equation}
    \label{monot}
     L[tv]\equiv Qt^2-Bt^3+Dt^4>0\quad \text{for} \quad t>0  \quad \Longleftrightarrow \quad B^2 - 4 Q D < 0.
    \end{equation}
    Then
    \begin{equation}
    E[v] = \frac Q 2 - \frac B 3 + \frac D 4 > \frac Q 2 - \frac 2 3 \sqrt{Q D} + \frac D 4 
    = Q \bigg ( \left ( \frac 2 3 - \frac 1 2 \sqrt{D/Q}\right)^2+\frac 1 {18}\bigg)\ge\frac {Q} {18},
    \end{equation}
    and the first claim follows.

    \medskip
    
    2. We start with observing that
    \begin{gather}
     0 = L[v] = Q - B + D, \qquad
     0 \ge H[v] = 2 Q - 3 B + 4 D.
    \end{gather}
    Thus 
    \begin{equation}
     \label{B_equals_Q_plus_D}
     B = Q + D; \qquad Q \ge D > 0.
    \end{equation}

    We should prove that for any $t \in [0,1]$
    \begin{equation}
     \frac {Qt^2}2 - \frac {Bt^3}3 + \frac {Dt^4}4 \ge \frac {Q t^2}{12}.
    \end{equation}
    Using the first relation in \eqref{B_equals_Q_plus_D}, we rewrite this inequality as follows:
    \begin{equation}
     \label{psi_definition}
     \psi(t) := \frac {\frac Q 2 t^2 - \frac B 3 t^3 + \frac D 4 t^4}{Q t^2 / 12} = 6 - 4 \left ( 1 + \frac D Q \right ) t + 3 \frac D Q t^2\ge 1.
    \end{equation}
    The second relation in \eqref{B_equals_Q_plus_D} gives
    \begin{gather}
     \psi'(t) = - 4 \left( 1 + \frac D Q \right)  +6 \frac D Q t \le  -4 + 2 \frac D Q \le -2, \qquad t\in[0,1],
    \end{gather}
    thus
    \begin{equation}
     \psi(t) \ge \psi(1) = 2 - \frac D Q \ge 1,
     \qquad t\in[0,1],
    \end{equation}
    and \eqref{psi_definition} follows.
   \end{proof}

   The next lemma shows that the Nehari manifold is separated from zero. Moreover, the ``end of ridge'' is separated from zero uniformly with respect to $\beta$.
   
   \begin{lemma}
    \label{Nehari_sets_away_from_zero_lemma}
    \begin{enumerate}
     \item If $v \in \mathcal N(\Omega)$, then $\|v\|_{L_4} \ge \frac{S_2(\Omega) S_4(\Omega)}{\beta}$.
     \item If $v \in \mathcal N_0(\Omega)$, then $\|v\|_{L_4} \ge S_4(\Omega)$.
    \end{enumerate}
   \end{lemma}
   \begin{proof}

  1. By the Cauchy inequality,
    \begin{equation}
     \beta \int\limits_\Omega v^3 \,dx \le \varepsilon \int\limits_\Omega v^2 \,dx + \frac 1 {4 \varepsilon} \beta^2 \int\limits_\Omega v^4 \,dx, \qquad \varepsilon>0.
    \end{equation}
    Then for $v \in \mathcal N(\Omega)$
    \begin{equation}
     0 = L[v] \ge Q[v] - \varepsilon \int\limits_\Omega v^2 \,dx - \frac 1 {4 \varepsilon} \beta^2 \int\limits_\Omega v^4 \,dx + \int\limits_\Omega v^4 \,dx.
    \end{equation}
    Setting $\varepsilon = \frac{S_2^2}2$, we have
    \begin{equation}
     0 \ge \frac 1 2 Q[v] + \frac 1 2 \Bigl[ Q[v] - S_2^2 \int\limits_\Omega v^2 \,dx \Bigr] - \frac {\beta^2}{2 S_2^2} \int\limits_\Omega v^4 \,dx + \int\limits_\Omega v^4 \,dx.
    \end{equation}
    The expression in square brackets is nonnegative, so
    \begin{equation}
     \frac{\beta^2 - 2 S_2^2}{2 S_2^2} \int\limits_\Omega v^4 \,dx \ge \frac 1 2 Q[v].
    \end{equation}
    As $\beta > 2 \sqrt{1-\alpha}$, using Remark \ref{S_two_bounds} we see that
    \begin{equation}
     \beta^2 - 2 S_2^2 > 4(1-\alpha) - 2(1-\alpha) = 2(1-\alpha) > 0.
    \end{equation}
    The embedding theorem implies
    \begin{equation}
     \frac{\beta^2 - 2 S_2^2}{S_2^2}\, \|v\|_{L_4}^4 \ge S_4^2 \|v\|_{L_4}^2.
    \end{equation}
   Thus
    \begin{equation}
     \|v\|_{L_4} \ge \frac {S_2 S_4}{\sqrt{\beta^2 - 2 S_2^2}}.
    \end{equation}
    The claimed estimate follows since $\sqrt{\beta^2 - 2 S_2^2} < \beta$.
    \medskip

    2. Observe that $L[v] = H[v] = 0$ is equivalent to
    \begin{equation}
     Q[v] = \int\limits_\Omega v^4 \,dx = \frac 1 2 \beta \int\limits_\Omega v^3 \,dx.
    \end{equation}
    Then
    \begin{equation}
     \|v\|_{L_4}^4 = \int\limits_\Omega v^4 \,dx = Q[v] \ge S_4^2 \|v\|_{L_4}^2,
    \end{equation}
and the second claim follows.   
    \end{proof}
   
   \begin{corollary}
    \label{rear_slope_weakly_closed_lemma}
    The set $\mathcal M_-(\Omega)$ is weakly closed in ${\cal W}(\Omega)$.
   \end{corollary}
   \begin{proof}
    Let $\mathcal M_-(\Omega) \ni u_n \rightharpoondown u$. Since $L$ and $H$ are weakly lower semicontinuous and $K_1$ is weakly continuous, inequalities $L[u] \le 0$, $H[u] \le 0$ and $K_1[u] \ge 0$ hold. Finally, since ${\cal W}(\Omega)$ is compactly embedded into $L_4(\Omega)$, we obtain that 
    \begin{equation}
     \|u\|_{L_4} = \lim \|u_n\|_{L_4} \ge \frac {S_2 S_4} \beta > 0.
    \end{equation}
    Thus, $u \not= 0$, and the statement follows.
   \end{proof}

   \begin{lemma}
    \label{pullback_lemma}
    Let $u \in \mathcal M_-(\Omega)$. Then there exists $t^* \in (0,1]$ such that $t^* u \in \mathcal N_-(\Omega) \cup \mathcal N_0(\Omega)$. Moreover, $K_0[t^* u] \leqslant K_0[u]$.
   \end{lemma}
   \begin{proof}
    If $ u \in \mathcal N_-(\Omega) \cup \mathcal N_0(\Omega)$, then the statement is obvious with $t^*=1$.
    Otherwise we have $L[u] < 0$, so fibration $E[t u]$ has a nondegenerate maximum $t^* \in (0,1)$, that is $t^* u \in \mathcal N_-(\Omega)$. As $K_1[u] \ge 0$, we obtain $K_1[tu] >0$ for all $t \in (0,1)$. As
    \begin{equation}
     \frac {d}{dt}(K_0[tu]) = \frac 1 t K_1[tu],
    \end{equation} 
    we obtain that $K_0[tu]$ strictly increases on $(0,1] \ni t$, and $K_0[t^* u] < K_0[u]$. 
   \end{proof}

The last statement in this subsection gives us an explicit formula for the energy of functions on the ``ridge''. Consider the functional 
   \begin{equation}
    I[v]= \frac { \int\limits_\Omega v^3 \,dx }{\Big (Q[v] \int\limits_\Omega v^4 \,dx\Big )^{\frac 1 2}}.   
   \end{equation}

   \begin{lemma}
    \label{uniformization_lemma}
    Let $v \in {\cal W}(\Omega)$ be such that $\int\limits_\Omega v^3 \,dx > 0$.
    Then:
    \begin{enumerate}
     \item the fibration $\varphi_v(t) = E[t v]$ is 
     \begin{itemize}
         \item monotonous if $I[v] < \frac 2 \beta$;
         \item degenerate if $I[v] = \frac 2 \beta$;
         \item non-monotonous if $I[v] > \frac 2 \beta$.
     \end{itemize}
     
     \item for any $\beta > 2 \left ( I[v] \right )^{-1}$ there exists $\tilde t=\tilde t(v,\beta)$ such that $\tilde t v \in \mathcal N_-(\Omega)$.
     \item \begin{equation}
          \label{uniform_energy}
          E[\tilde t v] = \frac 1 {3 \beta^2} \cdot \frac {(Q[v])^3}{\Big ( \int\limits_\Omega v^3 \,dx \Big )^2} \cdot f\bigg (\frac {\beta I[v]} 2 \bigg ),
         \end{equation}
           where
           \begin{equation}
            f(s) = \frac {1 + 3 \sqrt{1 - s^{-2}}}{\left ( 1 + \sqrt{1 - s^{-2}} \right )^3}.
           \end{equation}
     \item $f(s)$ decreases on $[1,+\infty) \ni s$; $f(1) = 1$; $\lim_{s \to +\infty} f(s) = \frac 1 2$.
    \end{enumerate}
   \end{lemma}
   Proof is given in the Appendix 2. Notice that the right-hand side in \eqref{uniform_energy} is homogeneous with respect to $u$.
   
  \begin{corollary}
  \label{convex}
    The Nehari manifold for the problem \eqref{Swift--Hohenberg-equation}--\eqref{Navier-2} is empty if and only if 
    \begin{equation}
     \beta<\beta_* = 2 \,\Big(\sup_{v \in \cal W} I[v]\Big)^{-1}.  
    \end{equation}
In particular, if $\beta<2 S_2(\Omega)$ then $u\equiv0$ is a unique solution of \eqref{Swift--Hohenberg-equation}--\eqref{Navier-2}.
  \end{corollary}

\begin{proof}
    The first statement immediately follows from the first claim in Lemma \ref{uniformization_lemma}. The second one follows from the Cauchy--Bunyakovsky inequality: 
   \begin{equation}
      I[v]\le \frac {\Big(\int_\Omega u^2 dx \int_\Omega u^4 dx\Big)^{\frac 12}}{\Big(Q[v] \int_\Omega v^4 dx\Big)^{\frac 12}} \le S_2^{-1}.
    \end{equation}
\end{proof}

\begin{remark}
    Two-sided estimates of $S_2(\Omega)$ in terms of the parameter $\alpha$ are given in Remark \ref{S_two_bounds}.
\end{remark}

  \subsection{Nehari method of obtaining a solution}
   
   Now we can deal with minimization problem:
   \begin{lemma}
    \label{Nehari_minimum_attained_lemma}
    Functional $E$ attains its minimum on $\mathcal N_-(\Omega) \cup \mathcal N_0(\Omega)$.
   \end{lemma}
   \begin{proof}
    By Lemma \ref{energy_coercive_lemma} functional $E$ is coercive, thus any minimizing sequence $u_n\in \mathcal N_-(\Omega) \cup \mathcal N_0(\Omega)$ must be bounded in ${\cal W}(\Omega)$. Then without loss of generality we can assume that $u_n \rightharpoondown u$. By Corollary \ref{rear_slope_weakly_closed_lemma} we have $u\in \mathcal M_-(\Omega)$.
    
    We claim that in fact $u\in \mathcal N_-(\Omega) \cup \mathcal N_0(\Omega)$. Indeed, otherwise we can see from the proof of Lemma \ref{pullback_lemma} that there is $t^* \in (0,1)$ such that $t^* u\in \mathcal N_-(\Omega)$ and $K_0[t^*u]<K_0[u]$. Since $K_0$ is weakly continuous (Lemma \ref{functionals_continuity_lemma}), this implies
    \begin{equation}
     \lim_{n \to \infty} E[u_n] =\lim_{n \to \infty} \left(E[u_n]-\frac 12 L[u_n] \right) = \lim_{n \to \infty} K_0[u_n]=K_0[u] > K_0[t^* u] = E[t^* u] - \frac 1 2 L[t^* u] = E[t^* u],
    \end{equation}
that is impossible since $u_n$ is a minimizing sequence. Thus, the claim follows, and $u$ is the minimizer we are looking for.
   \end{proof}
   
   From this point on $U$ will denote the minimizer constructed in Lemma \ref{Nehari_minimum_attained_lemma}.
   
   \begin{lemma}
    \label{negative_curvature_lemma}
    For all $\beta >2 S_3^3 S_4^{-2}$ the strict inequality $H[U] < 0$ holds or, equivalently, $U\not\in \mathcal N_0(\Omega)$.
   \end{lemma}
   
   \begin{proof}
    Let $\hat u$ be the minimizer of \eqref{Sobolev_constant_Neumann} with $m=3$. By part 1 of Lemma \ref{uniformization_lemma} for any $\beta > 2 \left ( I[\hat u] \right )^{-1}$ there exists $\tilde t >0$ such that $\tilde t \hat u \in \mathcal N_-(\Omega)$. Notice that 
    \begin{equation}
     (I[\hat u])^{-1} = \frac {\Big(Q[\hat u] \int_\Omega \hat u^4 dx\Big)^{\frac 12}}{\int_\Omega \hat u^3 dx} = S_3^3\,\frac {\Big(Q[\hat u] \int_\Omega \hat u^4 dx\Big)^{\frac 12}}{\big ( Q[\hat u] \big )^{\frac 3 2}}\le S_3^3 S_4^{-2}.
    \end{equation}
Further, parts 2 and 3 of Lemma \ref{uniformization_lemma} show that 
    \begin{equation}
     E[\tilde t \hat u] = \frac 1 {3 \beta^2}\ S_3^6 \cdot f\bigg(\frac {\beta I[\hat u]}2\bigg)\le \frac {S_3^6} {3 \beta^2}.
    \end{equation}

    As $U$ is the minimizer of  $E$ on $\mathcal N_-(\Omega) \cup \mathcal N_0(\Omega)$,
    \begin{equation}
     \label{energy_b_growth_order}
     E[U] \le E[\tilde t \hat u] \le \frac {S_3^6} {3 \beta^2}.
    \end{equation}

    Assume now that $U \in \mathcal N_0(\Omega)$. Then by Lemma \ref{coercitivity_at_zero_lemma} and part 2 of Lemma \ref{Nehari_sets_away_from_zero_lemma} we obtain
    \begin{equation}
     \label{ridge_side_energy_lower_estimate}
     E[U] \ge \frac 1 {12} Q[U] \ge \frac {S_4^2} {12} \|U\|_{L_4}^2 \ge \frac {S_4^4} {12}.
    \end{equation}
    This contradicts \eqref{energy_b_growth_order} if $\beta > {2 S_3^3} {S_4^{-2}}$. 
   \end{proof}

   \begin{theorem}
    \label{Nehari_solution_exists_theorem}
    Let $\beta > \beta_0:= 2 \max\{\sqrt{1-\alpha}, S_3^3 S_4^{-2} \}$. Then the minimizer $U$ is a weak solution of the boundary value problem \eqref{Swift--Hohenberg-equation}--\eqref{Navier-2}.
   \end{theorem}
   \begin{proof}
    As $U$ is the minimizer of $E[u]$ with constraints $L[u]=0$ and $H[u] \le 0$, there exist Lagrange multipliers $\lambda_1$ and $\lambda_2$ such that for all $h \in {\cal W}(\Omega)$ the identity
    \begin{equation}
     \label{Lagrange_identity}
     dE[U]h - \lambda_1 \cdot dL[U]h - \lambda_2 \cdot dH[U]h =0
    \end{equation}
    holds. By Lemma \ref{negative_curvature_lemma} condition $H[U] \le 0$ is inactive and thus $\lambda_2=0$.
    
    Next, substituting $h=U$ into \eqref{Lagrange_identity}, we obtain
    \begin{equation}
     0 = dE[U]U - \lambda_1 \cdot dL[U]U = L[U] - \lambda_1 \cdot H[U] = - \lambda_1 \cdot H[U].
    \end{equation}
    Since $H[U] < 0$, we obtain that $\lambda_1 = 0$ and thus $U$ is a critical point of $E[u]$ on ${\cal W}(\Omega)$.
   \end{proof}

   We call $U$ {\bf the ridge-Nehari solution} of \eqref{Swift--Hohenberg-equation}--\eqref{Navier-2}.

   \begin{remark}
   \label{b-bounded}
    The quantity $\beta_0$ in Theorem \ref{Nehari_solution_exists_theorem} depends on $\alpha$ and $\Omega$. However, Lemma \ref{Sobolev_constants_bounded_lemma} shows that, given $\Omega$, we have
    \begin{equation}
    \label{frak-b}
    \beta^*:=\sup\limits_{R\ge 1} \, \beta_0(\alpha,\Omega_R)<\infty.
    \end{equation}
   \end{remark}
   
  \subsection{Irreducibility of the ridge-Nehari solutions in $\Omega_R$ for large $R$}\label{ss:irred}

In this subsection, $\mathcal B_r$ stands for the open ball with radius $r$ centered in the origin.
  
   \begin{lemma}
    \label{lem:irreducibility_test_function}
       For any $\beta>2\sqrt{1-\alpha}$ there exists $r = r(\alpha,\beta)>0$ and a function $v \in \mathcal C_0^\infty(\mathcal B_r)$ such that the fibration $E[t v]$ is non-monotonous. 
   \end{lemma} 
   \begin{proof}
    Consider a radial smooth function $v$ such that
    \begin{equation}
        v(x) = \begin{cases}
                1,& |x| \le r-1;\\
                0,& |x| \ge r - \frac 1 2.
               \end{cases}
    \end{equation}
    Then we have, as $r\to\infty$,
    \begin{equation}
    Q[v]=(1-\alpha)|\mathcal B_r| (1 + O(r^{-1}));\quad
    \int\limits_{\mathcal B_r} v^3 \,dx = |\mathcal B_r| (1 + O(r^{-1}));\quad 
    \int\limits_{\mathcal B_r} v^4 \,dx = |\mathcal B_r| (1 + O(r^{-1})), 
    \end{equation}
    and therefore, if $r$ is large enough,
    \begin{equation}
     I[v] = (1-\alpha)^{-\frac 1 2} (1 + O(r^{-1}))>\frac 2\beta.  
    \end{equation}
    Now the statement follows from part 1 of Lemma \ref{uniformization_lemma}.
   \end{proof}

Without loss of generality we can assume that $0 \in \Omega$. As we consider expanding domains, we can assume also that $\mathcal B_1\subset \Omega$.
   
   \begin{lemma}
    \label{lem:irreducibility_upper_bound}
    Let $\beta > \beta^*$, where $\beta^*$ was introduced in \eqref{frak-b}. 
    Let $U_R$ be the ridge-Nehari solution in $\Omega_R$. Then the energy $E[U_R]$ is uniformly bounded with respect to $R\ge1$. 
   \end{lemma}
   
   \begin{proof}
    For $R > r$ the ball $\mathcal B_r$ lies inside $ \Omega_R$. Consider the test function $v \in \mathcal C_0^\infty(\mathcal B_r)$ introduced in Lemma \ref{lem:irreducibility_test_function}. By Lemma \ref{uniformization_lemma} there exists $\tilde t(\beta) > 0$ such that $\tilde t(\beta) v \in \mathcal N_-(\Omega_R)$. Obviously, $E[\tilde t(\beta)v]$ does not depend on $R$. As $E[U_R] \le E[\tilde t(\beta) v]$, the claim follows.
   \end{proof}

   \begin{corollary}
   Under assumptions of Lemma \ref{lem:irreducibility_upper_bound},
   there exists $R_0 = R_0(\alpha, \beta,\Omega)$ such that the ridge-Nehari solution is not constant for $R > R_0$.   
   \end{corollary}
   \begin{proof}
    By Lemma \ref{uniformization_lemma}
    \begin{equation}
     E[c_-] = \frac 1 {3 \beta^2} \frac {(1-\alpha)^3|\Omega_R|^3}{|\Omega_R|^2} \cdot f\bigg(\frac {\beta I[c_-]} 2 \bigg) \ge \frac {(1-\alpha)^3}{6 \beta^2}\, |\Omega| R^n.   
    \end{equation}   
    However, by Lemma \ref{lem:irreducibility_upper_bound} $E[U_R] \le C(\alpha,\beta)$, and the statement follows.
   \end{proof}

   \begin{lemma}
    \label{lem:irreducibility_lower_bound}
    Let $\Omega=\omega\times\omega'$ be a Carthesian product of two domains $\omega \subset \mathbb{R}^k$ and $\omega' \subset \mathbb R^{n-k}$. 
    Let $C_0 > 0$ and let $u_R$ be arbitrary weak solution of the boundary value problem \eqref{Swift--Hohenberg-equation}--\eqref{Navier-2} in $\Omega_R$ such that $E[u_R] \le C_0$.
    If $u_R$ does not depend on $x_1, \dots, x_k$ then there exists $R_1=R_1(\alpha, \beta, C_0, \omega, \omega')$ such that $u_R \not \in \mathcal N_-(\Omega_R)$ for $R>R_0$.
   \end{lemma}
   
   \begin{proof}
    We have
    \begin{equation}
     u_R(x_1,x_2, \ldots, x_n) = \mathfrak u_R(x_{k+1}, \ldots, x_n).
    \end{equation}
    Then the profile function $\mathfrak u_R$ is a weak solution of \eqref{Swift--Hohenberg-equation}--\eqref{Navier-2} in $\omega'_R$ and therefore lies in Nehari manifold $\mathcal N(\omega'_R)$. By Lemma \ref{Nehari_sets_away_from_zero_lemma}, we have $\|\mathfrak u_R\|_{L_4} \ge \frac{S_2(\omega'_R) S_4(\omega'_R)}{\beta}$. 

    If $u_R \in \mathcal N_-(\Omega_R)$ then 
    \begin{equation}
    H[u_R]<0 \quad\Longrightarrow\quad H[\mathfrak u_R]=\frac {H[u_R]}{|\omega|R^k}<0,
    \end{equation}
    that is, $\mathfrak u_R \in \mathcal N_-(\omega'_R)$. So, Lemmata \ref{coercitivity_at_zero_lemma} and \ref{Sobolev_constants_bounded_lemma} imply
    \begin{equation}
     E[\mathfrak u_R]\ge \frac 1{12}\, Q[\mathfrak u_R]\ge \frac {S^2_4(\omega'_R)}{12}\, \|\mathfrak u_R\|^2_{L_4} \ge \frac{S^2_2(\omega'_R) S^4_4(\omega'_R)}{12\beta^2}\ge \frac{C(\alpha,\omega')}{\beta^2}, \qquad R\ge1.
    \end{equation}
    On the other hand,
    \begin{equation}
     E[\mathfrak u_R] = \frac {E[u_R]} {|\omega|R^k}  \le \frac {C_0}{|\omega|R^k},
    \end{equation}
    and we reach a contradiction for large $R$.
   \end{proof}
   
   Lemmata \ref{lem:irreducibility_upper_bound} and \ref{lem:irreducibility_lower_bound} demonstrate that the ridge-Nehari solutions of \eqref{Swift--Hohenberg-equation}--\eqref{Navier-2} in $\Omega_R$ essentially depend on all variables if $R$ is large enough.

\begin{remark}
    Notice that the quantity $\beta^*$ is always greater than $2$. We conjecture that for small $|\alpha|$ the assumption $\beta>\beta^*$ is too strong, and that for large $R$ the ridge-Nehari solutions in $\Omega_R$ exist and are irreducible for $\beta>{\rm const}\cdot\sqrt{-\alpha}$.
\end{remark}

 \section{Construction of entire solutions with special symmetries}
  \label{S:periodic_solutions}

  The proof of the following lemma is standard. We give it for the reader's convenience.
 
  \begin{lemma}
   \label{lem:reflection_extension}
   Let $\Omega$ have a face $F$, let $\mathcal R$ be a reflection with respect to $F$, and let $u$ be a weak solution of the boundary value problem \eqref{Swift--Hohenberg-equation}--\eqref{Navier-2} in $\Omega$. Consider the even extension
   \begin{equation}
    \widetilde u(x) = \begin{cases}
            u(x), x \in \Omega;\\
            u(\mathcal Rx), x \in \mathcal R\Omega.
           \end{cases}   
   \end{equation}
   Then $\widetilde u$ is a weak solution of \eqref{Swift--Hohenberg-equation}--\eqref{Navier-2} in the ``doubled'' domain $\widetilde\Omega=\Omega \cup F\cup \mathcal R \Omega$.
  \end{lemma}
  \begin{proof}
  First, the boundary condition $\frac {\partial u}{\partial \bf{n}} = 0$ on $F$ implies $\widetilde u\in {\cal W}(\widetilde\Omega)$. Next, consider a test function $h \in {\cal W}(\widetilde\Omega)$. It can be represented as
   \begin{equation}
    h(x) = h_1(x)+h_2(x):=\frac {h(x) + h(\mathcal R x)} 2 + \frac {h(x) - h(\mathcal R x)} 2.   
   \end{equation}
   Then $h_1(x)=h_1(\mathcal R x)$, and the integral identity \eqref{weak_SH_solution} immediately gives $dE[\widetilde u]h_1 =0$. The equality $dE[\widetilde u]h_2 =0$ holds because all terms are odd with respect to $F$. So, $dE[\widetilde u]h=0$.
   \end{proof}

  Now we assume that the polyhedron $\Omega$ has the following property: the space $\mathbb R^n$ can be filled with reflections of $\Omega$, colored checkerwise. Following \cite{NSch}, we call such a polyhedron the {\bf fundamental domain}.\footnote{Fundamental domains in $\R^2$ are rectangles, equilateral triangles, isosceles right triangles and right triangles with acute angle $\frac \pi 6$. On the other hand, the right hexagon is not a fundamental domain since the hexagonal tiling of the plane cannot be colored checkerwise. In dimension $n\ge2$, for instance, we can take as $\Omega$ the Cartesian product of fundamental domains in spaces of lower dimensions.}

  \begin{theorem}
   Let $\Omega$ be a fundamental domain in $\R^n$, $n\leq 7$. Let $u$ be a weak solution of the boundary value problem \eqref{Swift--Hohenberg-equation}--\eqref{Navier-2} in $\Omega$. Then repeated even extension of $u$ gives a function $\bf u$ in $\R^n$, and $\bf u$ is a classical solution of \eqref{Swift--Hohenberg-equation}.
  \end{theorem}
  \begin{proof}
   By Lemma \ref{lem:reflection_extension}, $\bf u$ is a generalized solution of \eqref{Swift--Hohenberg-equation} in $\R^n$. By classical elliptic regularity theory, it is a classical solution.
  \end{proof}

Next, we can begin with the ridge-Nehari solution of \eqref{Swift--Hohenberg-equation}--\eqref{Navier-2} in a fundamental domain $\Omega_R$ instead of $\Omega$. As is shown in Subsection~\ref{ss:irred}, for $R$ large enough we obtain nontrivial (depending on all variables) solutions. We stress that solutions for different (sufficiently large) $R$ cannot be obtained from each other by coordinate dilation and multiplying by a constant (since the equation \eqref{Swift--Hohenberg-equation} is not invariant with respect to such
transformations).

In general, different domains $\Omega$ can give the same solution. For instance, we conjecture that solutions in $\mathbb R^2$ generated by isosceles right triangle coincide with corresponding solutions generated by the square. However, beginning with the ridge-Nehari solutions in rectangles $R\times a R$ with mutually incommensurable aspect ratios $a \ge 1$, we obviously obtain different solutions in $\mathbb R^2$. Solutions with a different periodic structure are generated by equilateral triangles.

Different nontrivial solutions in $\mathbb R^3$ can be obtained from ridge-Nehari solutions in right parallelepipeds with various ratios of sides and in right triangular prisms. Similarly one can consider the cases $4\le n \le 7$.

 \section{Skew-periodic solutions}
  \label{S:skew_periodic_solutions}
 
  Let $\Omega \subset {\mathbb R}^2$ be a parallelogram spanned by vectors $h_1, h_2$. Consider the periodic conditions
  \begin{equation}
   \label{periodic_boundary_conditions}
   u(x+h_k) = u(x), \quad \frac{\partial u}{\partial {\bf n}}(x+h_k) = \frac{\partial u}{\partial {\bf n}}(x),  \quad k=1,2,
  \end{equation}
  and the space ${\cal W}_P \subset W_2^2(\Omega)$ of functions satisfying \eqref{periodic_boundary_conditions}. Stationary points of energy functional \eqref{energy-functional} satisfy the integral identity \eqref{weak_SH_solution} with test functions in ${\cal W}_P$ and are weak solutions of \eqref{Swift--Hohenberg-equation} with periodic boundary conditions, comprised of conditions \eqref{periodic_boundary_conditions} and (periodic) natural boundary conditions. Notice that any weak solution of \eqref{Swift--Hohenberg-equation} -- \eqref{periodic_boundary_conditions} can be considered as a function on a flat torus, and any shift of this function on a torus is also a weak solution.

  Constant functions lie in ${\cal W}_P$, so the results for constant solutions can be used as proven in subsection \ref{ss:constant}. The equivalence of $\sqrt{Q[u]}$ and standard norm in $W_2^2(\Omega)$ can be proven essentially by the argument in proof of Lemma \ref{equivalent_norm_lemma}. The ridge-Nehari solutions in ${\cal W}_P$ are then obtained and their properties are explored with exactly the same arguments as in ${\cal W}$. Specifically, we prove that the ridge-Nehari solutions in $\Omega_R$ are irreducible to functions of less variables (effectively depend on all variables) if $R$ is large enough. Similar to Lemma \ref{lem:reflection_extension}, we prove that periodic extension of such solutions are generalized solutions to \eqref{Swift--Hohenberg-equation} in whole ${\mathbb R}^2$ and, by regularity argument, are classical solutions in the plane. 
  

  For two pairs of generating vectors $(h_1, h_2)$ and $(h_1', h_2')$, denote by $M$ the corresponding transition matrix, that is $(h_1', h_2') = (h_1, h_2) M$. If both $M$ and $M^{-1}$ have integer entries 
  then lattices generated by $(h_1, h_2)$ and $(h_1', h_2')$ coincide. Then any weak solution of \eqref{Swift--Hohenberg-equation} -- \eqref{periodic_boundary_conditions} in $\Omega$, extended periodically to the whole plane and then restricted to $\Omega'$, is a weak solution of \eqref{Swift--Hohenberg-equation} -- \eqref{periodic_boundary_conditions} in $\Omega'$. Moreover, all integrals over $\Omega$ appearing in energy functional $E[u]$ coincide with corresponding integrals over $\Omega'$, so any ridge-Nehari solution in $\Omega$, after extension and restriction to $\Omega'$, generates a ridge-Nehari solution in $\Omega'$, and vice versa.

  Now we give a sufficient condition which guarantees that solutions in two parallelograms generate essentially different entire solutions.

  \begin{theorem}
  \label{irrat}
  Let the transition matrix $M$ have the following property: for arbitrary orthogonal matrix $V$, at least one entry of the matrix $MV$ is irrational. Then entire solutions generated by the ridge-Nehari solutions in parallelograms $R\Omega$ and $R\Omega'$ with $R$ large enough are essentially different, that is, cannot be obtained one from another by rotations and shifts.   
  \end{theorem}

  \begin{proof}
   The claim is equivalent to following: if transition matrix $M$ has at least one irrational entry, then entire solutions generated by the ridge-Nehari solutions in parallelograms $R\Omega$ and $R\Omega'$ with $R$ large enough cannot be obtained one from another by shifts.
   
   Assume without loss of generality that $h_1' = a (h_1 + b h_2)$, with $a \not \in \mathbb Q$. Consider the image of sublattice $H = (a, ab) \mathbb Z$ on a torus $[0,1] \times [0,1]$ with opposite sides identified.  The projection of $H$ to the first coordinate is a dense set. If $b$ is also irrational, $H$ is dense in the whole torus, and then $u$ is a constant. If $b$ is rational, $H$ is a dense subset of a certain straight line on the torus, and $u$ is constant in a certain direction. Both outcomes are not possible for the ridge-Nehari solutions if $R$ is large enough.     
  \end{proof}

  In the space of dimension $3\le n\le 7$, given any parallelepiped $\Omega$, one can construct a corresponding skew-periodic entire solution, and a natural analog of Theorem \ref{irrat} holds true.

  \section{Appendix}
   
   \subsection{Proof of Lemma \ref{equivalent_norm_lemma}}
    \label{S:proof_equivalent_norm_lemma}
   
    Let $v_j$ be a family of eigenfunctions of $-\Delta$ with Neumann boundary condition orthonormal in $L_2(\Omega)$. Let $\mu_j$ be corresponding eigenvalues. Obviously,
    \begin{equation}
    \label{int-v^2}
     \int\limits_\Omega (D v_i, D v_j) \,dx = \int\limits_\Omega  v_i(- \Delta v_j) \,dx = \mu_j \delta_{ij}.
    \end{equation}
    Since $\lbrace v_j \rbrace$ is a complete orthonormal system, any function $u \in {\cal W}(\Omega)$ admits a decomposition
    \begin{equation}
     u = c_j v_j, \quad\mbox{where}\quad c_j = \int\limits_\Omega u v_j \,dx.
    \end{equation}
    Then
    \begin{equation}
     \int\limits_\Omega u^2 = \sum_j c_j^2 \quad\mbox{and}\quad \int\limits_\Omega |Du|^2 = \sum_j \mu_j c_j^2.
    \end{equation}
    
    To derive a similar decomposition for $\int\limits_\Omega |D^2 u|^2$, we prove the following lemma which is of independent interest.

\begin{lemma}
Let $\Omega$ be a convex polyhedron. Then $v_i \in W_2^2 (\Omega)$, and 
    \begin{equation}
    \label{int-D2v^2}
        \int\limits_\Omega (D^2 v_i, D^2 v_j) \,dx=\mu_j^2\delta_{ij}.
    \end{equation}
\end{lemma}

\begin{proof}
It is well known that the the functions $v_i$ are smooth outside of a neighborhood of the edges and vertices of $\Omega$. We introduce a smoothed-out convex domain $\Omega_\varepsilon\subset\Omega$ such that the boundary $\partial\Omega_\varepsilon$ can be split into two parts: $\Gamma_\varepsilon$ lies in $\varepsilon$-neighborhood of the edges and vertices, whereas $\partial\Omega_\varepsilon\setminus\Gamma_\varepsilon\subset\partial\Omega$, see Fig.~\ref{fig:1}.


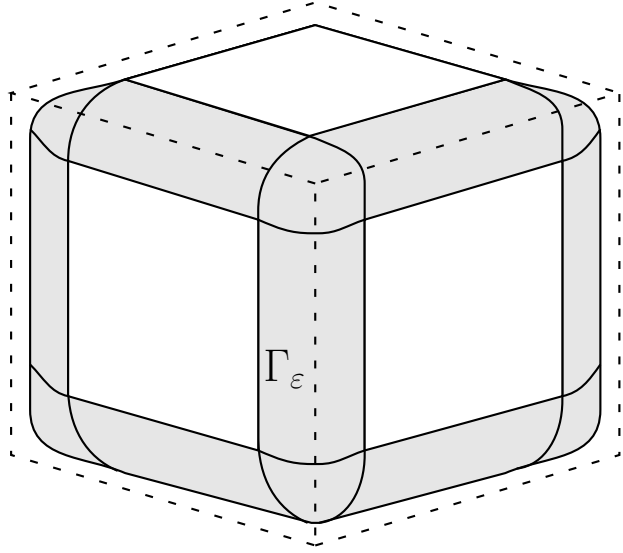
\begin{figure}[h]
    \centering
\begin{tikzpicture}[xscale=0.5,yscale=0.6]

\filldraw [fill=black!10, thick]
(0,11.5) to (-5,10.3) to [out=196,in=90] (-7.5,9) to [out=270,in=90] (-7.5,3) to [out=-90,in=160] (-5,1.6) to (-0.1,0.5) to (0,0.5) to (0.1,0.5) to (5,1.6) to [out=20, in=-90] (7.5,3) to (7.5,9) to [out=90, in=-15] (5,10.3) to (0,11.5) ;

\draw [thick] (0,11.5) to (-5,10.3) to [out=196,in=90] (-6.5,8.5) to [out=270,in=90] (-6.5,3.5) to [out=-90,in=160] (-5.2,1.65);

\draw [thick, xshift=5cm,yshift=-1.1cm] (0,11.4) to (-5,10.2) to [out=200,in=90] (-6.5,8.5) to [out=270,in=90] (-6.5,3.5) to [out=-90,in=160] (-5.2,1.6);

\draw [thick]
(-5,10.3) to (0,9) to [out=-20, in=90] (1.3,8) to [out=270,in=90] (1.3,2) to [out=-90,in=5] (0.1,0.5);

\draw [thick, xshift=5.2cm,yshift=1.2cm]
(-0.2,9.1) to [out=-20, in=90] (1.3,8) to [out=270,in=90] (1.3,2) to [out=-90,in=5] (0.1,0.5);

\draw [thick] (-7.5,4) to [out=-40, in=170] (-6.5,3.3);
\draw [thick] (-1.5,2.1) to [out=-20,in=180] (-0.1,1.8) [out=0,in=180] to (0.1,1.8) to [out=0, in=195] (1.3,2.1);
\draw [thick] (6.5,3.3) to [out=10,in=-130] (7.5,4);

\begin{scope}[yshift=5.2cm]
\draw [thick] (-7.5,4) to [out=-40, in=170] (-6.5,3.3);
\draw [thick] (-1.5,2) to [out=-20,in=180] (-0.1,1.7) [out=0,in=180] to (0.1,1.7) to [out=0, in=195] (1.3,2);
\draw [thick] (6.5,3.3) to [out=10,in=-130] (7.5,4);
\end{scope}

\filldraw[fill=white, draw=black, thick] (0,11.5) -- (-5,10.3) -- (-0.15,9.05) -- (5,10.3) -- cycle;

\filldraw[fill=white, draw=black, thick] (-6.5,8.5) -- (-6.5,3.3) -- (-1.5,2.1) -- (-1.5,7.2) -- cycle;

\filldraw[fill=white, draw=black, thick] (6.5,8.5) -- (6.5,3.3) -- (1.3,2.1) -- (1.3,7.2) -- cycle;

\draw[thick,loosely dashed] (0,0)--(8,2)--(8,10)--(0,8)--cycle;
\draw[thick,loosely dashed] (0,8)--(-8,10)--(0,12)--(8,10);
\draw[thick,loosely dashed] (-8,10)--(-8,2)--(0,0);

\draw (-0.8,3.9) node[] {\Large$\Gamma_\varepsilon$};

\end{tikzpicture}

\caption{The domain $\Omega_\varepsilon$ for a cube $\Omega$ in $\mathbb R^3$}
    \label{fig:1}
\end{figure}

Consider the integral
  \begin{equation}
   \int\limits_{\Omega_\varepsilon} (D^2 v_i, D^2 v_j)\,\,dx = \int\limits_{\Omega_\varepsilon}\sum\limits_{kl}D_kD_lv_iD_kD_lv_j\,dx,
  \end{equation}
Integrating by parts, we obtain
    \begin{equation}
    \aligned
\int\limits_{\Omega_\varepsilon} (D^2 v_i, D^2 v_j)\,\,dx = - \int\limits_{\Omega_\varepsilon}\sum\limits_{kl}D_kv_iD_kD_lD_lv_j\,dx \ + & \int\limits_{\partial\Omega_\varepsilon\setminus\Gamma_\varepsilon} \sum\limits_{kl}D_kv_iD_kD_lv_j{\bf n}_l\, dS_x\\
+ & \int\limits_{\Gamma_\varepsilon} \sum\limits_{kl}D_kv_iD_kD_lv_j{\bf n}_l\, dS_x=: I_1+I_2+I_3.
    \endaligned
    \end{equation}

By \eqref{int-v^2}, we have, as $\varepsilon\to0$,
$$
I_1=\int\limits_{\Omega_\varepsilon} (Dv_i, D(-\Delta v_j)) \,dx = \mu_j \int\limits_{\Omega_\varepsilon} (Dv_i, Dv_j) \,dx \to \mu_j^2 \delta_{ij}.
$$    

Next, we observe that $I_2$ is a sum of integrals over flat faces, where ${\bf n}$ is constant. Therefore,
    \begin{equation}
     I_2=\int\limits_{\partial\Omega_\varepsilon\setminus\Gamma_\varepsilon} \sum\limits_k D_kv_iD_k\frac {\partial v_j}{\partial\bf n} \, dS_x=\int\limits_{\partial\Omega_\varepsilon\setminus\Gamma_\varepsilon} \Big[\sum\limits_k\partial_kv_i\partial_k\frac {\partial v_j}{\partial\bf n} + \frac {\partial v_i}{\partial\bf n} \frac {\partial^2 v_j}{\partial{\bf n}^2}\Big] \, dS_x,
    \end{equation}
where $\partial_k=D_k-{\bf n}_k\frac {\partial}{\partial{\bf n}}$ is the tangential differential operator. Due to the Neumann boundary conditions, we have $I_2=0$.

Dealing with $I_3$, we use the estimates in \cite[Theorem 4.3]{Escobar} and \cite{Mazya}. They imply that
\begin{equation}
    |Dv_i|\le C,\quad |D^2v_i|\le C\varepsilon^{\lambda-2} \quad \mbox{on} \ \ \Gamma_\varepsilon, 
\end{equation}
where $\lambda>1$ depends only on $\Omega$, and $C$ depends on $v_i$ but does not depend on $\varepsilon$. Since $|\Gamma_\varepsilon|=O(\varepsilon)$, we obtain $I_3\to0$ as $\varepsilon\to0$. This completes the proof.  \end{proof}    

\begin{remark}
    Formula \eqref{int-D2v^2} holds also for eigenfunctions and eigenvalues of the Dirichlet Laplacian in a convex polyhedron. The proof runs without essential changes.
\end{remark}

We continue the proof of Lemma \ref{equivalent_norm_lemma}. Formula \eqref{int-D2v^2} immediately gives for $u \in {\cal W}(\Omega)$
    \begin{equation}
     \int\limits_\Omega (D^2 u, D^2 u) \,\,dx = \sum_{ij} c_i c_j \int\limits_\Omega (D^2 v_i, D^2 v_j) \,dx = \sum_j \mu_j^2 c_j^2. 
    \end{equation}
    Thus
    \begin{equation}
     \|u\|_{W_2^2}^2=
     \int\limits_\Omega \left [ |D^2u|^2 + |Du|^2 + u^2 \right ]
     =\sum_j (\mu_j^2 + \mu_j +1 )\, c_j^2.
    \end{equation}
    In a similar way,
    \begin{equation}
     Q[u] = \sum_j ((\mu_j-1)^2 - \alpha)\, c_j^2.
    \end{equation}

    Observe that the ratio
    \begin{equation}
     \frac {(\mu-1)^2 - \alpha}{\mu^2 + \mu +1 }
    \end{equation}
    tends to $1$ as $|\mu| \to \infty$ and is continuous and strictly positive (recall that $\alpha<0$). By compactness argument, we have
    \begin{equation}
     c (\mu^2 + \mu +1) \le (\mu-1)^2 - \alpha \le C(\mu^2 + \mu +1),
    \end{equation}
    and the Lemma follows.
\hfill$\square$

  \subsection{Proof of Lemma \ref{uniformization_lemma}}
   \label{S:proof_uniformization_lemma}
   We use the notations $Q$, $B$ and $D$ introduced in \eqref{QBD}.

   To prove the first claim, we recall that by \eqref{monot} $\varphi_v(t)$ is monotonous if and only if $B^2 - 4QD<0$. Observing that $I[v] = \frac {B}{\beta \sqrt{QD}}$, we obtain that
    \begin{equation}
     I^2[v] = \frac {B^2} {\beta^2 QD} = \frac {B^2 - 4QD} {\beta^2 QD} + \frac 4 {\beta^2} < \frac 4 {\beta^2}.
    \end{equation}
    The non-monotonous and degenerate cases are considered similarly. This also proves the second claim.

For $\beta > 2 I[v]^{-1}$, the equation $L[tv] = 0$ has two positive roots:
   \begin{equation}
    t = \frac {B \pm \sqrt{B^2 - 4 Q D}}{2 D} = \frac{\sqrt{Q}}{\sqrt{D}} \cdot (s \pm \sqrt{s^2 - 1}),
   \end{equation}
   (here $s=\frac {\beta I[v]} 2>1$), and the ``ridge'' root $\tilde t$ corresponds to the ``minus'' sign.

   Now we calculate the energy:
   \begin{equation}
    E[\tilde t v] = E[\tilde t v] - \frac 1 4 L[\tilde t v] = \frac 1 {12}\, \tilde t\,^2 (3Q - B\tilde t) = \frac 1{12D}\, (B\tilde t-Q)(3Q - B\tilde t).
   \end{equation}
Since 
   \begin{equation}
    \frac BQ\,\tilde t = 2s(s - \sqrt{s^2 - 1}) = \frac {2s} {s + \sqrt{s^2 - 1}},
   \end{equation}
we have
   \begin{equation}
   \aligned
   E[\tilde t v] = \frac {Q^2}{12D} \Big ( \frac {2s}{s+ \sqrt{s^2 - 1}} - 1 \Big) \Big ( 3 - \frac {2s} {s + \sqrt{s^2 - 1}} \Big ) 
    = & \ \frac {Q^2}{12D}\cdot \frac {(s- \sqrt{s^2 - 1})(s+ 3\sqrt{s^2 - 1})}{(s+ \sqrt{s^2 - 1})^2}\\
   = &\  \frac {Q^2}{12D}\cdot \frac {s+ 3\sqrt{s^2 - 1}}{(s+ \sqrt{s^2 - 1})^3} = \frac {Q^2}{12Ds^2}\cdot f(s).
   \endaligned
   \end{equation}
Since
   \begin{equation}
    \frac {Q^2}{12Ds^2} = \frac {Q^3}{3B^2} = \frac {(Q[v])^3}{3 \beta^2 \Big( \int\limits_\Omega v^3 \,dx \Big)^2},
  \end{equation}
the third claim follows. The last claim is proved by elementary calculus.
\hfill$\square$


\end{document}